\documentclass[11pt]{amsart}

\usepackage{amsmath,amssymb,amsthm}
\usepackage[margin=1.1in]{geometry}
\usepackage{enumitem}
\usepackage[colorlinks=true,linkcolor=blue,citecolor=blue,urlcolor=blue]{hyperref}

\newtheorem{theorem}{Theorem}[section]
\newtheorem{lemma}[theorem]{Lemma}
\newtheorem{proposition}[theorem]{Proposition}
\newtheorem{corollary}[theorem]{Corollary}
\theoremstyle{definition}

\newtheorem{example}[theorem]{Example}
\newtheorem{remark}[theorem]{Remark}
\newtheorem{question}[theorem]{Question}

\numberwithin{equation}{section}

\newcommand{\R}{\mathbb{R}}
\newcommand{\C}{\mathbb{C}}
\newcommand{\Q}{\mathbb{Q}}
\newcommand{\Z}{\mathbb{Z}}
\newcommand{\N}{\mathbb{N}}
\newcommand{\K}{\mathbb{K}}
\newcommand{\F}{\mathcal{F}}
\newcommand{\Fp}{\mathcal{F}^{+}}
\newcommand{\Fm}{\mathcal{F}^{-}}
\newcommand{\M}{\mathrm{M}_2}
\newcommand{\GL}{\mathrm{GL}}
\newcommand{\SL}{\mathrm{SL}}
\newcommand{\PGL}{\mathrm{PGL}}

\newcommand{\cl}[1]{\overline{#1}}
\newcommand{\mat}[4]{\begin{pmatrix} #1 & #2\\ #3 & #4\end{pmatrix}}
\newcommand{\smat}[4]{\left(\begin{smallmatrix} #1 & #2\\ #3 & #4\end{smallmatrix}\right)}
\newcommand{\gen}[1]{\langle #1\rangle}
\DeclareMathOperator{\diag}{diag}
\DeclareMathOperator{\diam}{diam}

\makeatletter
\@namedef{subjclassname@2020}{\textup{2020} Mathematics Subject Classification}
\makeatother

\begin{document}

\title{Dense $2$-generator subsemigroups of $2\times 2$ matrices \\ (revised and extended version)}

\author{Mohammad Javaheri}
\address{Department of Mathematics and Statistics, Siena University\\
515 Loudon Road, Loudonville, NY 12211, USA}
\email{mjavaheri@siena.edu}


\begin{abstract}
We show that the semigroup of real linear fractional transformations of a proper
subinterval of the real line does not contain a dense subsemigroup generated by two
elements, while it does contain dense subsemigroups generated by three elements. For the
semigroup generated by the two maps $x\mapsto 1+a/x$ and $x\mapsto x/b$ (the only pair
not excluded by elementary arguments) we compute the closure exactly, and we describe all of
its orbit closures in the plane. Consequently the semigroup of $2\times2$ matrices with
nonnegative entries has no dense $2$-generator subsemigroup but has dense $3$-generator
ones. In contrast, we prove a general density criterion for subsemigroups of $\M(\R)$ and
$\M(\C)$ generated by three explicit matrices, and use a word identity of length seven to
produce, for both $\R$ and $\C$, uncountable families of pairs of matrices that generate
dense subsemigroups; a fully explicit rational example is
$\smat{1}{1/2}{1}{0}$ and $\smat{1}{0}{0}{-8/3}$.
Thus the minimal number of generators of a dense subsemigroup is $3$ for the nonnegative
cone and $2$ for $\M(\R)$ and $\M(\C)$.
\end{abstract}

\maketitle

\section{Introduction}

Let $\F$ denote the set of real linear fractional transformations (LFTs) of the form
\[
f(x)=\frac{\alpha x+\beta}{\gamma x+\delta},\qquad \alpha,\beta,\gamma,\delta\ge 0,\quad
\alpha\delta-\beta\gamma\neq0 .
\]
Each $f\in\F$ maps $(0,\infty)$ into itself, and $\F$ is a semigroup under composition,
isomorphic to the semigroup of invertible $2\times2$ matrices with nonnegative entries
modulo multiplication by positive scalars. We say that $f_n\to f$ in $\F$ if
$f_n(x)\to f(x)$ for every $x>0$; this agrees with the natural topology coming from
$\PGL(2,\R)$ (Remark~\ref{rem:topology}).

In \cite{J1} we determined all pairs $f,g\in\F$ that generate a topologically transitive
semigroup on $(0,\infty)$, i.e.\ pairs for which the orbit of (almost) every point under
$\gen{f,g}$ is dense. Here we are interested in minimally generated \emph{dense}
subsemigroups of $\F$, and of related matrix semigroups. In \cite{J2} the following
question was asked.

\begin{question}\label{q:main}
What is the least number of generators of a dense subsemigroup of $\M(\K)$, the set of
$n\times n$ matrices over $\K\in\{\R,\C\}$?
\end{question}

The semigroup generated by a single matrix is never dense; it does not even have a dense
orbit in $\K^n$ (no linear operator on a finite-dimensional space is hypercyclic; see
\cite{Kitai,Rolewicz}). The present paper answers Question~\ref{q:main} for $n=2$, and also
answers the analogous question for the nonnegative cone.

For comparison we recall some results on dense \emph{subgroups}. A subgroup of $\SL(2,\C)$
is called elementary if the commutator of any two of its elements of infinite order has
trace $2$, and discrete if it contains no convergent sequence of distinct elements.
J{\o}rgensen \cite{Jorgensen} showed that a non-elementary subgroup of $\SL(2,\C)$ is
discrete if and only if each of its two-generator subgroups is discrete, that a
non-elementary subgroup of $\SL(2,\R)$ is discrete if and only if each of its cyclic
subgroups is discrete, and that every dense subgroup of $\SL(2,\R)$ contains a dense
two-generator subgroup. Wang \cite{Wang} proved that every dense subgroup of the group of
orientation-preserving M\"obius transformations of $S^n$, $n\ge2$, contains a dense subgroup
generated by at most $n$ elements, and Cao \cite{Cao} obtained a similar statement for
$U(n,1)$. Dense subsemigroups have been studied much less, although the semigroup structure
is in many respects the more natural one, and one may hope that results on dense subgroups
can be upgraded to statements about dense subsemigroups.

\subsection*{Main results}
\begin{enumerate}[leftmargin=*,label=(\arabic*)]
\item $\F$ contains no dense subsemigroup generated by two elements
(Corollary~\ref{cor:no2}). The only pair not ruled out by elementary arguments is
$R(x)=1+a/x$, $S(x)=x/b$, and for this pair we compute the closure $\cl{\gen{R,S}}$ exactly
(Theorem~\ref{thm:closure}): an increasing map $f=(\alpha x+\beta)/(\gamma x+1)$ lies in
$\cl{\gen{R,S}}$ if and only if the interval
$\bigl[\alpha-\beta\gamma,\ \alpha^2/(\alpha-\beta\gamma)\bigr]$ contains an integer power of $b$.
\item $\F$ contains dense subsemigroups generated by three elements; for instance
$\gen{1+a/x,\ x/b,\ x/c}$ is dense whenever $\ln c/\ln b\notin\Q$
(Theorem~\ref{thm:three}). Hence the minimal number of generators is exactly $3$.
\item The orbit closures of $\gen{R,S}$ acting diagonally on $(0,\infty)^2$ are described
explicitly (Theorem~\ref{thm:orbit}); the quantity $\kappa(x,y)=(x-y)^2/(xy)$ never
increases along orbits, so no orbit in $(0,\infty)^2$ is dense, whereas for the conjugate
action on $[0,1]^2$ the dense orbits are exactly those of the boundary points other than
$(0,0)$ and $(1,1)$ (Proposition~\ref{prop:square}).
\item The cone $\M(\R_{\ge0})$ of nonnegative $2\times2$ matrices has no dense
$2$-generator subsemigroup but has dense $3$-generator ones (Proposition~\ref{prop:nonneg2},
Theorem~\ref{thm:nonneg3}).
\item A general criterion (Theorem~\ref{thm:criterion}): for $\K\in\{\R,\C\}$ the matrices
$\smat{u}{a}{1}{0}$, $\diag(1,b)$, $\diag(c,1)$ generate a dense subsemigroup of $\M(\K)$
under explicit, generic conditions on $u,a,b,c$.
\item Two generators suffice for $\M(\R)$ and $\M(\C)$: for every $b,t,u$ with
$bt^2+(b+2)t+1=0$ and $u^5t(2t+1)=1$, the matrices $A=\smat{u}{u^2t}{1}{0}$ and
$B=\diag(1,b)$ satisfy $ABA^3BA=\diag(c,1)$ (Proposition~\ref{prop:word}). This yields a
one-parameter family of $2$-generator dense subsemigroups of $\M(\R)$, defined for all but
countably many parameter values (Theorem~\ref{thm:realfamily}), containing the rational
example above, and a family of $2$-generator dense subsemigroups of $\M(\C)$ indexed by
almost every $b\in\C$ with $|b|>3$ (Theorem~\ref{thm:complexfamily}).
\end{enumerate}

The paper is organized as follows. Section~\ref{sec:prelim} collects notation and three
elementary density lemmas. Section~\ref{sec:F} studies dense subsemigroups of $\F$,
Section~\ref{sec:orbits} describes orbit closures, and Section~\ref{sec:matrices} treats
matrix semigroups over $\R_{\ge0}$, $\R$ and $\C$. Section~\ref{sec:remarks} lists some
open problems.

\section{Preliminaries}\label{sec:prelim}

Throughout, $\K\in\{\R,\C\}$. For $X\in\GL(2,\K)$ we write $[X]$ for its image in
$\PGL(2,\K)=\GL(2,\K)/\K^{*}$, which acts on $\widehat{\K}=\K\cup\{\infty\}$ by
\[
[X](x)=\frac{\alpha x+\beta}{\gamma x+\delta},\qquad X=\mat{\alpha}{\beta}{\gamma}{\delta}.
\]
Then $[XY]=[X]\circ[Y]$. We use the matrices
\[
D(\lambda)=\mat{\lambda}{0}{0}{1},\qquad L(s)=\mat{1}{0}{s}{1},\qquad U(v)=\mat{1}{v}{0}{1},
\]
so that $[D(\lambda)](x)=\lambda x$, $[L(s)](x)=x/(sx+1)$ and $[U(v)](x)=x+v$. Note that
$L(s)L(s')=L(s+s')$, $U(v)U(v')=U(v+v')$ and $D(\lambda)D(\lambda')=D(\lambda\lambda')$.

For a subset $E$ of a semigroup, $\gen{E}$ denotes the semigroup of all finite nonempty
products of elements of $E$. Including the identity would change none of the results below.
By continuity of multiplication, the closure of a subsemigroup of $\M(\K)$ or of
$\PGL(2,\K)$ is again a semigroup; we use this repeatedly without comment.

\begin{remark}[Topology]\label{rem:topology}
$\PGL(2,\K)$ carries the quotient topology. Since the quotient map is open,
$[X_n]\to[X]$ if and only if there are $\lambda_n\in\K^*$ with $\lambda_nX_n\to X$;
equivalently, if $X_{ij}\neq0$ then $X_n/(X_n)_{ij}\to X/X_{ij}$. Since a M\"obius map is
determined by the images of three distinct points and depends continuously on them (by the
cross-ratio formula), $[X_n]\to[X]$ if and only if $[X_n](x)\to[X](x)$ at three distinct
points, if and only if this holds at every point of $\widehat\K$. In particular, for
$\F\subset\PGL(2,\R)$ the topology of pointwise convergence on $(0,\infty)$ coincides with
the subspace topology. Moreover $\F$ is closed in $\PGL(2,\R)$: if $\lambda_nX_n\to X$ with
$X_n\ge0$ entrywise and $X$ invertible, then the $\lambda_n$ have eventually constant sign,
so $X$ or $-X$ is nonnegative. All closures below are therefore taken in $\PGL(2,\K)$ or in
$\M(\K)$.
\end{remark}

We write $\Fp$ (resp.\ $\Fm$) for the increasing (resp.\ decreasing) maps in $\F$, i.e.\
those with $\det>0$ (resp.\ $<0$). Both are open and closed in $\F$, and $\F=\Fp\sqcup\Fm$.
If $f\in\Fp$ is represented by $\smat{\alpha}{\beta}{\gamma}{\delta}\ge0$ then $\delta>0$
(otherwise $\det=-\beta\gamma\le0$), so every $f\in\Fp$ has a unique representative with
$\delta=1$; we write
\begin{equation}\label{eq:Fabc}
F(\alpha,\beta,\gamma)(x)=\frac{\alpha x+\beta}{\gamma x+1},\qquad
\alpha,\beta,\gamma\ge0,\ \alpha-\beta\gamma>0,
\end{equation}
and set $d(f)=\alpha-\beta\gamma$ and $\sigma(f)=\alpha^2$. The coefficients
$\alpha,\beta,\gamma$ depend continuously on $f\in\Fp$. Note that $0<d(f)\le\alpha$, so
$d(f)^2\le\sigma(f)$.

\medskip
We shall need three elementary density statements.

\begin{lemma}\label{lem:kronecker}
Let $\lambda>0>\mu$ with $\mu/\lambda\notin\Q$. Then for every $L\in\N$ the set
$\{k\lambda+l\mu:\ k,l\in\N,\ k,l\ge L\}$ is dense in $\R$. The same holds if $k$ is
restricted to one residue class modulo $2$ and $l$ to the even integers.
\end{lemma}

\begin{proof}
Let $\rho=\mu/\lambda<0$, $X\in\R$ and $\varepsilon>0$. Since $\rho$ is irrational, the
sequence $(l\rho \bmod 1)_{l\ge L}$ is dense in $\R/\Z$, so there are arbitrarily large
$l\ge L$ with $\operatorname{dist}(X/\lambda-l\rho,\Z)<\varepsilon/\lambda$. Let $k$ be the
nearest integer to $X/\lambda-l\rho$; then $|k\lambda+l\mu-X|<\varepsilon$, and
$k\ge-l\rho+X/\lambda-1\to\infty$ as $l\to\infty$. For the parity statement apply the same
argument with $l=2l'$ and $k=2k'+e$, using density of $(l'\rho\bmod 1)$ to approximate
$(X/\lambda-e)/2$ modulo $1$.
\end{proof}

\begin{lemma}\label{lem:additive}
Let $w\in\K^*$ and $b\in\K$ with $|b|>1$. Assume $b<0$ if $\K=\R$, and $b\notin\R$ if
$\K=\C$. Then the additive semigroup generated by $\{wb^k:k\in\Z\}$ is dense in $\K$. If
$\K=\R$ and $b>1$, $w>0$, the same additive semigroup is dense in $[0,\infty)$.
\end{lemma}

\begin{proof}
Call a finite set $E\subset\K$ \emph{positively spanning} if every open half-space
$\{z:\operatorname{Re}(z\bar e)>0\}$, $|e|=1$, contains a point of $E$. If $E$ is positively
spanning, the convex cone generated by $E$ is $\K$ (a proper closed convex cone lies in a
closed half-space). We claim there is $M$ such that for every $k_0\in\Z$ the set
$E_{k_0}=\{wb^{k_0-j}:0\le j\le M\}$ is positively spanning. If $\K=\R$, take $M=1$: the
numbers $wb^{k_0}$ and $wb^{k_0-1}$ have opposite signs. If $\K=\C$, let
$\varphi'\in(0,\pi)$ with $\varphi'\equiv\pm\arg b\pmod{2\pi}$ (possible since $b\notin\R$).
The directions of $wb^{k_0-j}$, $j=0,1,\dots$, move around the circle in steps of angle
$\varphi'<\pi$, always in the same direction; once their total angular displacement exceeds
$2\pi$, which happens for $j\le M:=\lceil 2\pi/\varphi'\rceil$, every open arc of length
$\pi$ contains one of them. This proves the claim.

Now let $z\in\K$ and $\varepsilon>0$. Choose $k_0$ so negative that every element of
$E=E_{k_0}$ has modulus $<\varepsilon$, and write $z=\sum_{e\in E}t_e e$ with $t_e\ge0$.
Then $\sum_e\lfloor t_e\rfloor e$ lies in the additive semigroup (unless all $\lfloor
t_e\rfloor=0$, in which case $|z|<(M+1)\varepsilon$ and we may use a single element of
$E$), and its distance to $z$ is less than $(M+1)\varepsilon$. The last statement is
proved in the same way using only nonnegative multiples of $wb^{k}$, $k\to-\infty$.
\end{proof}

\begin{lemma}\label{lem:multiplicative}
Let $b,c\in\K^*$ with $|b|>1>|c|$, put $\rho=\ln|c|/\ln|b|$ and, if $\K=\C$,
$\beta=\arg b/2\pi$ and $\gamma=\arg c/2\pi$ (for any choice of arguments). Assume
\begin{itemize}
\item[] $(\K=\R)$\quad $b<0$ and $\rho\notin\Q$;
\item[] $(\K=\C)$\quad the numbers $1,\ \rho,\ \gamma-\beta\rho$ are linearly independent
over $\Q$.
\end{itemize}
Then for every $z\in\K$ and every $L\in\N$, $z$ is a limit of numbers $b^kc^l$ with
$k,l\ge L$.
\end{lemma}

\begin{proof}
Since $|b^kc^l|$ can be made arbitrarily small with $k,l\ge L$ (by Lemma~\ref{lem:kronecker}
applied to $\ln|b|,\ln|c|$), it suffices to treat $z\neq0$.

$\K=\R$. The sign of $b^kc^l$ is $(-1)^k$ when $l$ is even. By Lemma~\ref{lem:kronecker}
(with $k$ of the required parity and $l$ even) we can make $k\ln|b|+l\ln|c|$ arbitrarily
close to $\ln|z|$ with prescribed sign $(-1)^k=\operatorname{sign} z$.

$\K=\C$. Write $\lambda=\ln|b|>0$, $z=e^{X+2\pi i\Theta}$. Changing the choice of arguments
alters $\gamma-\beta\rho$ by an element of $\Z+\Z\rho$, so the hypothesis does not depend on
that choice. By Kronecker's theorem \cite[Ch.~23]{HW}, the hypothesis implies that the
sequence $l\mapsto\bigl(l\rho,\ l(\gamma-\beta\rho)\bigr)$, $l\ge L$, is dense in
$(\R/\Z)^2$. Choose $l\ge L$ such that, modulo $\Z$,
\[
l\rho\approx X/\lambda\qquad\text{and}\qquad l(\gamma-\beta\rho)\approx \Theta-X\beta/\lambda
\]
within $\varepsilon$, and let $k$ be the integer with
$|k-(X/\lambda-l\rho)|<\varepsilon$; as before $k\to\infty$ when $l\to\infty$. Then
$\ln|b^kc^l|=\lambda(k+l\rho)$ is within $\lambda\varepsilon$ of $X$ and
\[
k\beta+l\gamma=\bigl(X/\lambda-l\rho+\delta\bigr)\beta+l\gamma
= l(\gamma-\beta\rho)+X\beta/\lambda+\delta\beta,\qquad |\delta|<\varepsilon,
\]
is within $(1+|\beta|)\varepsilon$ of $\Theta$ modulo $\Z$. Hence $b^kc^l$ is close to $z$.
\end{proof}

\begin{remark}\label{rem:H2}
If $\arg b/2\pi=p/q\in\Q$, the condition in the complex case is equivalent to the linear
independence over $\Q$ of $1$, $\rho$ and $\gamma$. In particular for $b=ri$ it is the
condition $1,\ \ln|c|/\ln|b|,\ \arg c/2\pi$ linearly independent over $\Q$.
\end{remark}

\section{Dense subsemigroups of $\F$}\label{sec:F}

In this section we show that $\F$ has no dense $2$-generator subsemigroups and that it has
dense $3$-generator subsemigroups. We work on $(0,\infty)$; conjugating by a linear
fractional map, the results hold for LFTs of any proper subinterval of $\R$ (including half
lines and closed intervals).

Let $\theta(x)=1/(1+x)$. It maps $[0,\infty]$ homeomorphically onto $[0,1]$, and for
$f\in\F$ the map $\hat f=\theta f\theta^{-1}$ is a linear fractional self-map of $[0,1]$.
Since every $\hat f$ is monotone, $\diam\hat f([0,1])=|\hat f(1)-\hat f(0)|$.

\subsection{Reduction to one pair}

\begin{theorem}[{\cite{J1}}]\label{thm:classification}
Let $f,g\in\F$ be such that the orbit of almost every $x>0$ under $\gen{f,g}$ is dense in
$(0,\infty)$. Then, up to interchanging $f$ and $g$ and up to a simultaneous conjugation by
a map $x\mapsto ux^{v}$ with $u>0$, $v\in\{1,-1\}$, one of the following holds:
\begin{enumerate}[label=\textup{(\Roman*)}]
\item \textup{(i)} $f(x)=\dfrac{x}{x+a}$, $g(x)=bx+c$, with $a,b\ge1$, $c\ge0$, and $b>1$
if $c=0$; \quad \textup{(ii)} $f(x)=x/a$, $g(x)=bx$ with $a,b>1$, $\ln a/\ln b\notin\Q$;
\item \textup{(i)} $f(x)=\dfrac{a}{x+a}$, $g(x)=bx+c$, with $0\le c\le1$, $a>0$, $b\ge1$,
and $b>1$ if $c=0$; \quad \textup{(ii)} $f(x)=a/x$, $g(x)=bx+1$ with $a,b\ge1$;
\item $f(x)=\dfrac{a}{x+a}$, $g(x)=c+\dfrac{ab}{x}$ with $0\le c\le1$, $a>0$, $b\ge1$, and
$ab\le1$ if $c=0$.
\end{enumerate}
\end{theorem}

If $\gen{f,g}$ is dense in $\F$ then every orbit of $\gen{f,g}$ in $(0,\infty)$ is dense
(the $\F$-orbit of every point is $(0,\infty)$), so Theorem~\ref{thm:classification}
applies. Conjugation by $x\mapsto ux^{\pm1}$ is a homeomorphic automorphism of $\F$ and
preserves density. The following criterion disposes of most cases.

\begin{lemma}\label{lem:diameter}
Let $H\subseteq\F$ be a subsemigroup. Suppose there are $r<1$ and a finite set $E\subset H$
such that $\diam\hat h([0,1])\le r$ for all $h\in H\setminus E$. Then $H$ is not dense in
$\F$.
\end{lemma}

\begin{proof}
Let $p=(1-r)/4$ and $q=1-p$, so that $q-p>r$. The set
$O=\{h\in\F:|\hat h(q)-\hat h(p)|>r\}$ is open in $\F$ (evaluation is continuous) and
contains the identity. Since $\F$ has no isolated points, a dense subset of $\F$ meets $O$
in infinitely many points. But $O\cap H\subseteq E$ is finite.
\end{proof}

\begin{proposition}\label{prop:reduction}
In the situation of Theorem~\ref{thm:classification}, $\gen{f,g}$ is not dense in $\F$
except possibly in case \textup{(II)(i)} with $c=0$.
\end{proposition}

\begin{proof}
\emph{Case (I).} Both $f$ and $g$ are increasing, so $\gen{f,g}\subseteq\Fp$, and $\Fm$ is a
nonempty open subset of $\F$ disjoint from $\gen{f,g}$.

For the remaining cases we compute the conjugated maps on $[0,1]$. With
$\theta^{-1}(y)=(1-y)/y$ one finds:
\begin{align*}
f(x)=\frac{a}{x+a}:&\quad \hat f(y)=\frac{(a-1)y+1}{(2a-1)y+1},\quad \hat f([0,1])=[\tfrac12,1];\\
g(x)=bx+c:&\quad \hat g(y)=\frac{y}{(c-b+1)y+b},\quad \hat g([0,1])=[0,\tfrac{1}{c+1}];\\
f(x)=\frac ax:&\quad \hat f(y)=\frac{1-y}{(a-1)y+1},\quad \hat f(0)=1,\ \hat f(\tfrac12)=\tfrac1{a+1};\\
g(x)=c+\frac{ab}{x}:&\quad \hat g(y)=\frac{1-y}{(1+c)(1-y)+aby},\quad \hat g([0,1])=[0,\tfrac1{1+c}].
\end{align*}

\emph{Case (II)(ii).} Here $f(x)=a/x$ is an involution and $\hat g([0,1])=[0,\frac12]$
(with $c=1$). Every element of $\gen{f,g}$ other than $f$ is either of the form $g h$ or of
the form $fgh$ with $h\in\gen{f,g}\cup\{\mathrm{id}\}$. In the first case its image lies in
$[0,\frac12]$; in the second it lies in $\hat f([0,\frac12])=[\frac1{a+1},1]$. Lemma
~\ref{lem:diameter} applies with $r=\max(\frac12,\frac{a}{a+1})$ and $E=\{f\}$.

\emph{Case (II)(i) with $c>0$.} Every element has image in $[\frac12,1]$ or in
$[0,\frac1{c+1}]$; take $r=\max(\frac12,\frac1{1+c})<1$ and $E=\emptyset$.

\emph{Case (III).} If $c>0$, argue as in the previous case. If $c=0$, then $g(x)=ab/x$ is an
involution with $\hat g([0,1])=[0,1]$, but $\hat g(\frac12)=\frac1{1+ab}$. Every element
other than $g$ has the form $fh$ (image in $[\frac12,1]$) or $gfh$ (image in
$\hat g([\frac12,1])=[0,\frac1{1+ab}]$). Take $r=\max(\frac12,\frac1{1+ab})$ and
$E=\{g\}$.
\end{proof}

In the remaining case $f(x)=a/(x+a)$, $g(x)=bx$, $a>0$, $b>1$. Conjugating by $x\mapsto1/x$
and replacing $a$ by $1/a$ we obtain the pair
\begin{equation}\label{eq:RS}
R(x)=1+\frac ax=\Bigl[\mat{1}{a}{1}{0}\Bigr],\qquad
S(x)=\frac xb=\Bigl[\mat{1}{0}{0}{b}\Bigr]=[D(b^{-1})],\qquad a>0,\ b>1.
\end{equation}
For the rest of this section and in Section~\ref{sec:orbits}, $\Lambda=\gen{R,S}$ and
$\cl\Lambda$ is its closure (which lies in $\F$ by Remark~\ref{rem:topology}).

\subsection{The closure of $\gen{R,S}$}

\begin{lemma}\label{lem:basic}
$\cl\Lambda$ contains $D(b^k)$ for every $k\in\Z$, the involution $I(x)=a/x$, the maps
$L(s)$ for $s\ge0$, and the translations $U(v)$ for $v\ge0$.
\end{lemma}

\begin{proof}
For $n\ge1$ we have $RS^nR=\smat{1+ab^n}{a}{1}{a}$, hence
\begin{align*}
S^mRS^{l+m}R&=\Bigl[\mat{1+ab^{l+m}}{a}{b^m}{ab^m}\Bigr]
=\Bigl[\mat{b^{-m}+ab^{l}}{ab^{-m}}{1}{a}\Bigr]\\
&\xrightarrow[m\to\infty]{}
\Bigl[\mat{ab^l}{0}{1}{a}\Bigr]=L\bigl(\tfrac{b^{-l}}{a}\bigr)D(b^l)
\end{align*}
for every $l\in\Z$. Call this limit $E_l\in\cl\Lambda$. For $k\in\Z$ and $l>k$,
$E_lS^{l-k}=L(b^{-l}/a)D(b^k)\to D(b^k)$ as $l\to\infty$, so $D(b^k)\in\cl\Lambda$. Then
$L(b^{-l}/a)=E_lD(b^{-l})\in\cl\Lambda$ for all $l\in\Z$, and since $L$ is additive,
$L(s)\in\cl\Lambda$ for all $s$ in the additive semigroup generated by $\{b^{k}/a:k\in\Z\}$,
which is dense in $[0,\infty)$ by Lemma~\ref{lem:additive}. Hence $L(s)\in\cl\Lambda$ for
$s\ge0$.

Next, $S^mRS^m=[\smat{1}{ab^m}{b^m}{0}]=[\smat{b^{-m}}{a}{1}{0}]\to[\smat{0}{a}{1}{0}]=I$.
Finally $IL(t)I=[\smat{a}{a^2t}{0}{a}]=U(at)$, so $U(v)\in\cl\Lambda$ for $v\ge0$.
\end{proof}

\begin{lemma}\label{lem:U0}
Let $\alpha,\beta,\gamma\ge0$ with $0<\alpha-\beta\gamma\le\min(1,\alpha^2)$. Then
$F(\alpha,\beta,\gamma)\in\cl\Lambda$.
\end{lemma}

\begin{proof}
For $u,v,w\ge0$,
\[
L(u)U(v)L(w)=\mat{1+vw}{v}{u+w+uvw}{1+uv}\in\SL(2,\R),
\]
so $[L(u)U(v)L(w)]=F\bigl(\frac{1+vw}{1+uv},\frac{v}{1+uv},\frac{u+w+uvw}{1+uv}\bigr)$, a map
with $d=(1+uv)^{-2}$. If $\beta=0$ the hypothesis forces $\alpha=1$ and $F(1,0,\gamma)=
L(\gamma)\in\cl\Lambda$. If $\beta>0$, put $d=\alpha-\beta\gamma$ and
\[
u=\frac{1-\sqrt d}{\beta},\qquad v=\frac{\beta}{\sqrt d},\qquad w=\frac{\alpha-\sqrt d}{\beta}.
\]
The hypothesis $d\le\min(1,\alpha^2)$ is exactly $u,w\ge0$. Since $1+uv=1/\sqrt d$, the
first two normalized entries are $(1+vw)\sqrt d=\alpha$ and $v\sqrt d=\beta$, and the third
is determined by $d=\alpha-\beta\gamma$. Thus $F(\alpha,\beta,\gamma)=[L(u)U(v)L(w)]\in
\cl\Lambda$ by Lemma~\ref{lem:basic}.
\end{proof}

For $k\in\Z$ let
\begin{equation}\label{eq:Uk}
U_k=\bigl\{f\in\Fp:\ d(f)\le\min\bigl(b^k,\ b^{-k}\sigma(f)\bigr)\bigr\},\qquad
\mathcal U=\bigcup_{k\in\Z}U_k .
\end{equation}
Equivalently, $f\in U_k$ if and only if $b^k\in[d(f),\sigma(f)/d(f)]$; note that this
interval is nonempty because $d(f)^2\le\sigma(f)$.

\begin{theorem}\label{thm:closure}
Let $R,S$ be as in \eqref{eq:RS}. Then
\[
\cl\Lambda\cap\Fp=\mathcal U,\qquad \cl\Lambda\cap\Fm=I\circ\mathcal U,\qquad
\cl\Lambda=\mathcal U\cup I\circ\mathcal U .
\]
In other words, an increasing map $F(\alpha,\beta,\gamma)$ belongs to $\cl\Lambda$ if and only
if the interval $[\alpha-\beta\gamma,\ \alpha^2/(\alpha-\beta\gamma)]$ contains an integer power
of $b$, and a decreasing map $f$ belongs to $\cl\Lambda$ if and only if $I\circ f$ does.
\end{theorem}

\begin{proof}
\emph{$\mathcal U\subseteq\cl\Lambda$.} Lemma~\ref{lem:U0} says $U_0\subseteq\cl\Lambda$.
If $f=F(\alpha,\beta,\gamma)\in U_k$ then $D(b^{-k})f=F(b^{-k}\alpha,b^{-k}\beta,\gamma)$ has
$d=b^{-k}d(f)$ and $\sigma=b^{-2k}\sigma(f)$, so it lies in $U_0$. Hence
$U_k=D(b^k)U_0\subseteq\cl\Lambda$ by Lemma~\ref{lem:basic}.

\emph{$\mathcal U$ is a semigroup.} Let $f=F(\alpha,\beta,\gamma)\in U_k$ and
$g=F(u,v,w)\in U_l$. Then
\[
fg=\Bigl[\mat{\alpha u+\beta w}{\alpha v+\beta}{\gamma u+w}{\gamma v+1}\Bigr],\qquad
d(fg)=\frac{(\alpha-\beta\gamma)(u-vw)}{(\gamma v+1)^2}.
\]
Since $\gamma v+1\ge1$ we get $d(fg)\le b^{k}b^{l}$, and since $\alpha u\le\alpha u+\beta w$,
\[
d(fg)\le\frac{b^{-k}\alpha^2\,b^{-l}u^2}{(\gamma v+1)^2}\le
b^{-k-l}\Bigl(\frac{\alpha u+\beta w}{\gamma v+1}\Bigr)^2=b^{-k-l}\sigma(fg).
\]
Thus $U_kU_l\subseteq U_{k+l}$.

\emph{$\mathcal U$ is closed in $\Fp$.} Let $f_n\in U_{k_n}$ converge to $f\in\Fp$. Then
$d(f_n)\to d(f)>0$ and $\sigma(f_n)\to\sigma(f)$, so $b^{k_n}\in[d(f_n),\sigma(f_n)/d(f_n)]$
stays in a compact subset of $(0,\infty)$. Hence $(k_n)$ is bounded, some $k$ occurs
infinitely often, and passing to the limit gives $f\in U_k$.

\emph{$\Lambda\cap\Fp\subseteq\mathcal U$.} An element of $\Lambda$ is increasing if and only
if it contains an even number of factors $R$, and then it is a product of factors $S^k$
$(k\ge1)$ and $RS^kR$ $(k\ge0)$. Now $S^k=F(b^{-k},0,0)\in U_{-k}$, and
\[
RS^kR=\Bigl[\mat{1+ab^k}{a}{1}{a}\Bigr]=F\bigl(b^k+\tfrac1a,\ 1,\ \tfrac1a\bigr),
\]
which has $d=b^k$ and $\sigma=(b^k+1/a)^2\ge b^{2k}$, so it lies in $U_k$. Since $\mathcal U$
is a semigroup, $\Lambda\cap\Fp\subseteq\mathcal U$. As $\Fp$ is open and closed in $\F$,
$\cl\Lambda\cap\Fp=\cl{\Lambda\cap\Fp}\subseteq\mathcal U$.

\emph{Decreasing maps.} Since $I\in\cl\Lambda$ and $I^2=\mathrm{id}$, left composition with
$I$ is a bijection of $\cl\Lambda\cap\Fp$ onto $\cl\Lambda\cap\Fm$.
\end{proof}

\begin{corollary}\label{cor:linear}
A linear map $x\mapsto\lambda x$ belongs to $\cl{\gen{R,S}}$ if and only if
$\lambda\in b^{\Z}$. In particular $\gen{R,S}$ is not dense in $\F$.
\end{corollary}

\begin{proof}
For $f=F(\lambda,0,0)$ the interval in Theorem~\ref{thm:closure} is $\{\lambda\}$.
\end{proof}

\begin{corollary}\label{cor:no2}
$\F$ has no dense subsemigroup generated by two (or fewer) elements. The same holds for the
semigroup of linear fractional self-maps of any proper subinterval of $\R$.
\end{corollary}

\begin{proof}
Combine Theorem~\ref{thm:classification}, Proposition~\ref{prop:reduction} and Corollary
~\ref{cor:linear}. A single map never has a dense orbit (its orbits are monotone or
alternate between two monotone sequences), so one generator does not suffice either.
\end{proof}

\subsection{Three generators}

\begin{theorem}\label{thm:three}
Let $a>0$, $b>1$, $c>0$ with $\ln c/\ln b\notin\Q$. Then the semigroup generated by
\[
R(x)=1+\frac ax,\qquad S(x)=\frac xb,\qquad Q(x)=\frac xc
\]
is dense in $\F$. Consequently, the minimal number of generators of a dense subsemigroup of
$\F$ is $3$.
\end{theorem}

\begin{proof}
Let $\Lambda'=\gen{R,S,Q}\supseteq\Lambda$. Let $f=F(\alpha,\beta,\gamma)\in\Fp$ with
$\beta,\gamma>0$, and put $d=\alpha-\beta\gamma$; then $d<\alpha<\alpha^2/d$, so
$J=[d,\alpha^2/d]$ has nonempty interior. For $l\ge1$,
$g_l=D(c^l)f=F(c^l\alpha,c^l\beta,\gamma)$ satisfies $d(g_l)=c^ld$ and $\sigma(g_l)=
c^{2l}\alpha^2$, so by Theorem~\ref{thm:closure}, $g_l\in\cl\Lambda$ as soon as
$b^kc^{-l}\in J$ for some $k\in\Z$. Since $\ln c/\ln b\notin\Q$, the numbers
$-l\ln c$ $(l\ge1)$ are dense modulo $\ln b$, so the set $\{b^kc^{-l}:k\in\Z,l\ge1\}$ is
dense in $(0,\infty)$ and such $k,l$ exist. Then $f=Q^lg_l\in\cl{\Lambda'}$.

Maps with $\beta,\gamma>0$ are dense in $\Fp$, so $\Fp\subseteq\cl{\Lambda'}$; and
$\Fm=I\circ\Fp\subseteq\cl{\Lambda'}$ since $I\in\cl\Lambda$. The last statement follows from
Corollary~\ref{cor:no2}.
\end{proof}

\begin{remark}
In Theorem~\ref{thm:three} one may equally take $Q(x)=cx$; only the irrationality of
$\ln c/\ln b$ matters.
\end{remark}

\section{Orbit closures}\label{sec:orbits}

We now describe the orbit closures of $\Lambda=\gen{R,S}$ (with $R,S$ as in \eqref{eq:RS})
acting diagonally on $(0,\infty)^2$, $h(x,y)=(h(x),h(y))$. For a point $P=(x,y)$ with $x>y>0$
put
\[
p(P)=x-y,\qquad q(P)=\frac1y-\frac1x,
\]
\[
\Delta(P)=\Bigl\{(u,v)\in(0,\infty)^2:\ v\le u,\ \ u-v\le p(P),\ \ \frac1v-\frac1u\le q(P)\Bigr\}.
\]
Geometrically, $\Delta(P)$ is the region bounded by the diagonal, the half line of slope $1$
through $P$, and the branch of the hyperbola $1/v-1/u=q(P)$, which passes through $P$ and is
tangent to the diagonal at the origin. Let $I(x,y)=(a/x,a/y)$. For $P=(x,y)$ with $x\neq y$,
exactly one of $P$, $I(P)$ lies below the diagonal; call it $P^*$, and set
\[
\Omega(P)=\Delta(P^*)\cup I\bigl(\Delta(P^*)\bigr),\qquad
\widetilde\Omega(P)=\bigcup_{k\in\Z}\Omega(b^kP).
\]
For $x=y$ we put $\widetilde\Omega(P)=\{(t,t):t>0\}$. Since $(b^kP)^*=b^{\pm k}P^*$ and
$I$ interchanges the two parts, $\widetilde\Omega(P)=\bigcup_k\bigl(\Delta(b^kP^*)\cup
I\Delta(b^kP^*)\bigr)$. Finally, let
\[
\kappa(x,y)=\frac{(x-y)^2}{xy}=\frac xy+\frac yx-2 .
\]
Note that $\kappa(P)=p(P)q(P)$ for $x>y$, and that $\kappa$ is invariant under $I$ and under
$D(\lambda)$, $\lambda>0$.

\begin{lemma}\label{lem:reach}
Let $x>y>0$ and $(u,v)\in\Delta(x,y)$ with $u>v$. Then there is $f\in\cl\Lambda\cap\Fp$ with
$f(x)=u$ and $f(y)=v$.
\end{lemma}

\begin{proof}
We look for $f=F(1,\beta,\gamma)$, which belongs to $\cl\Lambda$ by Lemma~\ref{lem:U0} whenever
$\beta,\gamma\ge0$ and $\beta\gamma<1$. Solving $f(x)=u$, $f(y)=v$ gives
\[
\beta=\frac{uv(x-y)-xy(u-v)}{ux-vy},\qquad \gamma=\frac{(x-y)-(u-v)}{ux-vy},
\]
where $ux-vy>0$. The conditions $\beta\ge0$ and $\gamma\ge0$ are exactly
$1/v-1/u\le1/y-1/x$ and $u-v\le x-y$, i.e.\ $(u,v)\in\Delta(x,y)$. Moreover
\[
1-\beta\gamma=\frac{(u-v)(v+x)(x-y)(u+y)}{(ux-vy)^2}>0. \qedhere
\]
\end{proof}

\begin{theorem}\label{thm:orbit}
For every $P\in(0,\infty)^2$ the closure in $(0,\infty)^2$ of the orbit $\Lambda P$ is
$\widetilde\Omega(P)$. Explicitly, if $P=(x,y)$ with $x>y$, a point $(u,v)$ with $u>v$ is in
the orbit closure if and only if
\[
\exists\,k\in\Z:\qquad u-v\le b^k(x-y)\quad\text{and}\quad \frac1v-\frac1u\le b^{-k}\Bigl(\frac1y-\frac1x\Bigr),
\]
a point $(u,v)$ with $u<v$ is in it if and only if $(a/u,a/v)$ is, and the whole diagonal
is contained in it. In particular $\kappa\le\kappa(P)$ on the orbit closure of $P$; hence no
orbit of $\Lambda$ in $(0,\infty)^2$ is dense.
\end{theorem}

\begin{proof}
If $x=y$ the orbit lies on the diagonal and is dense there, since by Lemma~\ref{lem:basic}
$\cl\Lambda$ contains all translations $x\mapsto x+v$, $v\ge0$, and all maps $x\mapsto
b^kx$. Assume $x\neq y$; replacing $P$ by $I(P)$ (which is allowed since $I\in\cl\Lambda$ and
$I^2=\mathrm{id}$) we may assume $P=P^*$, i.e.\ $x>y$.

\emph{$\widetilde\Omega(P)\subseteq\cl{\Lambda P}$.} If $Q\in\Delta(b^kP)$ is off the diagonal,
Lemma~\ref{lem:reach} gives $f\in\cl\Lambda$ with $f(b^kP)=Q$, so $Q=(fD(b^k))(P)$ with
$fD(b^k)\in\cl\Lambda$; thus $Q\in\cl{\Lambda P}$, and so is $I(Q)$. A diagonal point $(t,t)$
is the limit of the points $(t+\varepsilon,t)\in\Delta(P)$ as $\varepsilon\downarrow0$.

\emph{$\cl{\Lambda P}\subseteq\widetilde\Omega(P)$.} Clearly $P\in\widetilde\Omega(P)$. We
show that $\widetilde\Omega(P)$ is invariant under $R$ and $S$ and closed in $(0,\infty)^2$.
Write $R=M\circ I$ with $M(x)=x+1$. The set $\widetilde\Omega(P)$ is $I$-invariant by
construction. Since $S=D(b^{-1})$ and $D(b^{-1})I=ID(b)$, $S$ maps $\Delta(b^kP)$ onto
$\Delta(b^{k-1}P)$ and $I\Delta(b^kP)$ onto $I\Delta(b^{k+1}P)$. For $u>v$ the translation $M$
preserves $u-v$ and decreases $1/v-1/u$, so $M\Delta(b^kP)\subseteq\Delta(b^kP)$. The set
$I\Delta(Q)$ ($Q$ below the diagonal) is described by $v-u\le a\,q(Q)$ and
$1/u-1/v\le p(Q)/a$ with $u\le v$, and is $M$-invariant for the same reason.

Each piece $\Delta(b^kP)$, $I\Delta(b^kP)$ is closed in $(0,\infty)^2$. Let $Q_n\to Q\in
(0,\infty)^2$ with $Q_n$ in a piece of index $k_n$. If $(k_n)$ is bounded, $Q$ lies in one of
finitely many closed pieces. If $k_n\to-\infty$ along a subsequence, then
$|u_n-v_n|\to0$ (for pieces $\Delta$) or $|1/u_n-1/v_n|\to0$ (for pieces $I\Delta$), and
similarly if $k_n\to+\infty$; in all cases $Q$ is on the diagonal, which is contained in
$\widetilde\Omega(P)$. Hence $\widetilde\Omega(P)$ is closed.

Finally, on $\Delta(b^kP)$ we have $\kappa(u,v)=(u-v)(1/v-1/u)\le b^kp(P)\,b^{-k}q(P)=
\kappa(P)$, and $\kappa$ is $I$-invariant.
\end{proof}

\begin{remark}
The function $\kappa$ depends only on the cross-ratio $x/y$ of the quadruple
$(0,\infty;x,y)$. It is invariant under $I$ and $S$ and strictly decreased by $M$ off the diagonal, so it
is non-increasing along every orbit of $\gen{R,S}$; this is the conceptual reason why no orbit
is dense in $(0,\infty)^2$. Theorem~\ref{thm:orbit} shows that $\kappa$ alone does not
determine orbit closures: the additional constraint is the existence of the exponent $k$.
\end{remark}

On a bounded interval the situation is different. As in Section~\ref{sec:F}, let
$\hat h=\theta h\theta^{-1}$ with $\theta(x)=1/(1+x)$; these maps act on $[0,1]$ and
diagonally on $[0,1]^2$.

\begin{proposition}\label{prop:square}
The orbit of $(x,y)\in[0,1]^2$ under $\gen{\hat R,\hat S}$ is dense in $[0,1]^2$ if and only if
$(x,y)$ lies on the boundary of $[0,1]^2$ and $(x,y)\notin\{(0,0),(1,1)\}$.
\end{proposition}

\begin{proof}
We have
\[
\hat R(y)=\frac{1-y}{2-2y+ay},\qquad \hat S(y)=\frac{by}{1-y+by},\qquad
\widehat{L(s)}(y)=\frac{s(1-y)+y}{(s+1)(1-y)+y}.
\]
$\hat R$ is a decreasing bijection of $[0,1]$ onto $[0,\frac12]$; $\hat S$ is an increasing
homeomorphism of $[0,1]$ fixing $0$ and $1$ with $\hat S^n(y)\to1$ for $y\in(0,1]$; and
$\widehat{L(s)}$ maps $[0,1]$ onto $[\frac s{s+1},1]$, with $\widehat{L(s)}(0)=\frac s{s+1}$.
By Lemma~\ref{lem:basic}, $\widehat{L(s)}$ lies in the closure of $\gen{\hat R,\hat S}$ for
$s\ge0$.

\emph{Necessity.} Interior points of $[0,1]^2$ correspond to points of $(0,\infty)^2$, whose
orbit closures are not dense by Theorem~\ref{thm:orbit}. The diagonal of $[0,1]^2$ is
invariant, so the orbits of $(0,0)$ and $(1,1)$ are not dense.

\emph{Sufficiency.} Let $\cl O(P)$ denote the orbit closure of $P$. If $Q\in\cl O(P)$ then
$\cl O(Q)\subseteq\cl O(P)$, by continuity of the maps. We first show that $(0,1)\in\cl O(P)$
for every boundary point $P\notin\{(0,0),(1,1)\}$. For $P=(0,y)$, $y\in(0,1]$:
$\hat S^n(P)=(0,\hat S^n(y))\to(0,1)$. For $P=(y,0)$, $y\in(0,1]$: $\hat S^n(P)\to(1,0)$ and
$\hat R(1,0)=(0,\frac12)$, which reduces to the previous case. For $P=(1,y)$, $y\in[0,1)$:
$\hat R(P)=(0,\hat R(y))$ with $\hat R(y)>0$. For $P=(y,1)$, $y\in[0,1)$:
$\hat R(P)=(\hat R(y),0)$ with $\hat R(y)>0$.

It remains to show that $\cl O=\cl O((0,1))=[0,1]^2$. Applying $\widehat{L(s)}$ to $(0,1)$
gives $(\frac s{s+1},1)\in\cl O$ for all $s\ge0$, so $[0,1]\times\{1\}\subseteq\cl O$.
Applying $\hat R$ gives $[0,\frac12]\times\{0\}\subseteq\cl O$, and applying $\hat S^n$ gives
$[0,\hat S^n(\frac12)]\times\{0\}$, hence $[0,1]\times\{0\}\subseteq\cl O$. Applying
$\widehat{L(s)}$ to this segment gives $[\frac s{s+1},1]\times\{\frac s{s+1}\}$, so
$\cl O$ contains the triangle $T=\{(x,y):1\ge x\ge y\ge0\}$. Next $\hat R(T)=
\{(x,y):0\le x\le y\le\frac12\}$, so $T\cup\hat R(T)\supseteq[0,1]\times[0,\frac12]$. Applying
$\hat S^n$ gives $[0,1]\times[0,\hat S^n(\frac12)]\subseteq\cl O$, and letting
$n\to\infty$ we obtain $\cl O=[0,1]^2$.
\end{proof}

\section{Dense subsemigroups of $2\times2$ matrices}\label{sec:matrices}

\subsection{From projective density to density}

\begin{lemma}\label{lem:lift}
Let $\mathcal S\subseteq\M(\K)$ be a semigroup and $\pi:\GL(2,\K)\to\PGL(2,\K)$ the
quotient map.
\begin{enumerate}[label=\textup{(\alph*)}]
\item If $\cl{\mathcal S}$ contains $dI$ for every $d\in\K$ and $\pi(\mathcal S\cap
\GL(2,\K))$ is dense in $\PGL(2,\K)$, then $\mathcal S$ is dense in $\M(\K)$.
\item If $\mathcal S\subseteq\M(\R_{\ge0})$, $\cl{\mathcal S}$ contains $dI$ for every
$d\ge0$, and $\pi(\mathcal S\cap\GL(2,\R))$ is dense in $\F$, then $\mathcal S$ is dense in
$\M(\R_{\ge0})$.
\end{enumerate}
\end{lemma}

\begin{proof}
(a) Let $X\in\GL(2,\K)$. By Remark~\ref{rem:topology} there are $F_n\in\mathcal S\cap\GL$ and
$\lambda_n\in\K^*$ with $\lambda_nF_n\to X$. Since $\lambda_nI\in\cl{\mathcal S}$ and
$\cl{\mathcal S}$ is a semigroup, $\lambda_nF_n\in\cl{\mathcal S}$, hence $X\in\cl{\mathcal
S}$. As $\GL(2,\K)$ is dense in $\M(\K)$, the claim follows.
(b) The same argument works, noting that if $X\ge0$ is invertible, then some entry
$X_{ij}>0$, and $(F_n)_{ij}\ge0$ forces $\lambda_n>0$ for large $n$. Invertible nonnegative
matrices are dense in $\M(\R_{\ge0})$ (consider $X+\varepsilon I$).
\end{proof}

\subsection{Nonnegative matrices}

\begin{proposition}\label{prop:nonneg2}
No two matrices generate a dense subsemigroup of $\M(\R_{\ge0})$.
\end{proposition}

\begin{proof}
Suppose $\gen{X,Y}$ is dense in $\M(\R_{\ge0})$. If $Z\ge0$ is invertible and $W_n\in
\gen{X,Y}$ converge to $Z$, then $W_n$ is invertible for large $n$ and $[W_n]\to[Z]$; hence
$\pi(\gen{X,Y}\cap\GL)$ is dense in $\F$. If $X$ and $Y$ are both invertible, this says
$\gen{[X],[Y]}$ is dense in $\F$, contradicting Corollary~\ref{cor:no2}. If exactly one of
them, say $X$, is singular, then $\gen{X,Y}\cap\GL=\{Y^n\}$ and $\gen{[Y]}$ would be dense in
$\F$, again a contradiction. If both are singular, $\gen{X,Y}$ consists of singular matrices,
and so does its closure.
\end{proof}

\begin{theorem}\label{thm:nonneg3}
Let $a>0$ and $b>1>c>0$ with $\ln c/\ln b\notin\Q$. Then the matrices
\[
\mat{c}{0}{0}{1},\qquad \mat{1}{0}{0}{b},\qquad \mat{1}{a}{1}{0}
\]
generate a dense subsemigroup of $\M(\R_{\ge0})$. Hence the minimal number of generators of a
dense subsemigroup of $\M(\R_{\ge0})$ is $3$.
\end{theorem}

\begin{proof}
Let $\mathcal S$ be the semigroup. Its image in $\F$ is generated by $x\mapsto cx=x/c^{-1}$,
$S(x)=x/b$ and $R(x)=1+a/x$, so it is dense in $\F$ by Theorem~\ref{thm:three}. For $k,l\ge1$,
\[
\diag(c^l,1)\mat{1}{a}{1}{0}\diag(1,b^k)=\mat{c^l}{ab^kc^l}{1}{0}.
\]
Given $d>0$, Lemma~\ref{lem:kronecker} (with $\lambda=\ln b$, $\mu=\ln c$) provides
$k_n,l_n\to\infty$ with $b^{k_n}c^{l_n}\to d/a$; then the matrices above converge to
$\smat{0}{d}{1}{0}$, whose square is $dI$. Thus $dI\in\cl{\mathcal S}$ for $d>0$, hence for
$d\ge0$. Lemma~\ref{lem:lift}(b) and Proposition~\ref{prop:nonneg2} complete the proof.
\end{proof}

\subsection{A density criterion over $\R$ and $\C$}

\begin{theorem}\label{thm:criterion}
Let $\K\in\{\R,\C\}$, $u,a\in\K^*$ and $b,c\in\K$ with $|b|>1>|c|>0$. Put
$\rho=\ln|c|/\ln|b|$ and, if $\K=\C$, $\beta=\arg b/2\pi$, $\gamma=\arg c/2\pi$. Assume
\begin{enumerate}[label=\textup{(H\arabic*)}]
\item $b<0$ if $\K=\R$, and $b\notin\R$ if $\K=\C$;
\item $\rho\notin\Q$ if $\K=\R$, and $1,\rho,\gamma-\beta\rho$ are linearly independent over
$\Q$ if $\K=\C$.
\end{enumerate}
Then the matrices
\[
A=\mat{u}{a}{1}{0},\qquad B=\mat{1}{0}{0}{b},\qquad C=\mat{c}{0}{0}{1}
\]
generate a dense subsemigroup of $\M(\K)$. More precisely, $\gen{[A],[B]}$ is already dense in
$\PGL(2,\K)$.
\end{theorem}

\begin{proof}
\emph{Step 1: $\Lambda=\gen{[A],[B]}$ is dense in $\PGL(2,\K)$.} Write $R=[A]$, $S=[B]=
[D(b^{-1})]$. As in Lemma~\ref{lem:basic},
\begin{align*}
S^mRS^{l+m}R&=\Bigl[\mat{u^2+ab^{l+m}}{ua}{ub^m}{ab^m}\Bigr]=
\Bigl[\mat{u^2b^{-m}+ab^l}{uab^{-m}}{u}{a}\Bigr]\\
&\longrightarrow
\Bigl[\mat{ab^l}{0}{u}{a}\Bigr]=L\Bigl(\frac{u\,b^{-l}}{a}\Bigr)D(b^l),
\end{align*}
because $|b|>1$. Exactly as in Lemma~\ref{lem:basic} we deduce that $\cl\Lambda$ contains
$D(b^k)$ for all $k\in\Z$, and $L(ub^{-l}/a)$ for all $l\in\Z$. By Lemma~\ref{lem:additive}
(with $w=u/a$) and (H1), the additive semigroup generated by $\{ub^k/a\}$ is dense in $\K$,
so $L(s)\in\cl\Lambda$ for every $s\in\K$. Moreover
$S^mRS^m=[\smat{ub^{-m}}{a}{1}{0}]\to[\smat0a10]=:I$, and $IL(t)I=U(at)$, so
$U(v)\in\cl\Lambda$ for every $v\in\K$. Now every $X\in\SL(2,\K)$ with $X_{12}\neq0$ is of the
form
\[
X=\mat{p}{q}{r}{s}=L\Bigl(\frac{s-1}{q}\Bigr)U(q)L\Bigl(\frac{p-1}{q}\Bigr),
\]
so $\cl\Lambda\supseteq\pi(\SL(2,\K))$. If $\K=\C$ this is all of $\PGL(2,\C)$. If $\K=\R$,
$\PGL(2,\R)$ is the union of $\pi(\SL(2,\R))$ and its coset of classes with negative
determinant; since $\det B=b<0$, the latter equals $S\cdot\pi(\SL(2,\R))\subseteq\cl\Lambda$.

\emph{Step 2: scalars.} For $k,l\ge1$,
\[
C^lAB^k=\mat{c^lu}{ab^kc^l}{1}{0}.
\]
Let $d\in\K$. By Lemma~\ref{lem:multiplicative} and (H1), (H2), there are $k_n,l_n\to\infty$
with $b^{k_n}c^{l_n}\to d/a$; since $|c|<1$, $C^{l_n}AB^{k_n}\to\smat0d10$, whose square is
$dI$. Hence $dI\in\cl{\gen{A,B,C}}$ for every $d\in\K$.

\emph{Step 3.} Apply Lemma~\ref{lem:lift}(a).
\end{proof}

\begin{remark}\label{rem:special}
(1) Taking $\K=\R$, $u=1$, and replacing $(b,c)$ by $(-b,-c)$ with $b>1>c>0$ and
$\ln c/\ln b\notin\Q$, we recover the statement that
$\smat{-c}{0}{0}{1}$, $\smat{1}{0}{0}{-b}$, $\smat{1}{a}{1}{0}$ generate a dense subsemigroup of
$\M(\R)$. (2) Taking $\K=\C$ and $b=ri$, $r>1$, condition (H2) is equivalent to the
$\Q$-linear independence of $1$, $\ln|c|/\ln|b|$ and $\arg c/2\pi$ (Remark~\ref{rem:H2}).
(3) The hypotheses are invariant under conjugation by diagonal matrices and transposition,
both of which preserve density; e.g.\ $\smat{u}{a}{e}{0}$ with $e\neq0$ may replace $A$.
\end{remark}

\subsection{Two generators}

The following identity shows that, for suitable parameters, the third generator $C$ in
Theorem~\ref{thm:criterion} is a word in $A$ and $B$.

\begin{proposition}\label{prop:word}
Let $\K$ be a field of characteristic zero, $u,a,b\in\K$, and
$A=\smat{u}{a}{1}{0}$, $B=\diag(1,b)$. Then $W=ABA^3BA$ has entries
\begin{align*}
W_{11}&=u^5+2(1+b)u^3a+b(b+2)ua^2, & W_{12}&=a\,W_{21},\\
W_{21}&=u^4+(b+2)u^2a+ba^2, & W_{22}&=ua(u^2+2a).
\end{align*}
Consequently, if $b\neq0$, $t\in\K$ satisfies $bt^2+(b+2)t+1=0$, $u^5\,t(2t+1)=1$ and
$a=u^2t$, then $t(2t+1)\neq0$ and
\[
ABA^3BA=\mat{c}{0}{0}{1},\qquad c=\frac{1+2(1+b)t+b(b+2)t^2}{t(2t+1)}.
\]
The number $c$ does not depend on the choice of the fifth root $u$, and it is a root of
\begin{equation}\label{eq:cquad}
z^2-(b^2-b+2)z+b=0 .
\end{equation}
If $b\neq1$ and $b^2+4\neq0$, the two roots $t$ give the two distinct roots of \eqref{eq:cquad},
whose product is $b$.
\end{proposition}

\begin{proof}
The entries of $W$ are obtained by direct multiplication, using
$A^3=\smat{u^3+2au}{au^2}{u^2+a}{au}$. With $a=u^2t$ we get
$W_{21}=u^4(1+(b+2)t+bt^2)=0$, $W_{22}=u^5t(1+2t)=1$ and
$W_{11}=u^5(1+2(1+b)t+b(b+2)t^2)=c$. If $t=0$ the quadratic gives $1=0$, and if
$t=-\frac12$ it gives $-b/4=0$; so $t(2t+1)\neq0$. Using $bt^2=-(b+2)t-1$ one rewrites
\[
c=\frac{b\bigl((b^2+2b+2)t+b+1\bigr)}{(b+4)t+2},
\]
and a direct computation shows that
\[
b\bigl((b^2+2b+2)t+b+1\bigr)^2-(b^2-b+2)\bigl((b^2+2b+2)t+b+1\bigr)\bigl((b+4)t+2\bigr)
+\bigl((b+4)t+2\bigr)^2
\]
equals $-b(1-b)^2\bigl(bt^2+(b+2)t+1\bigr)$. Multiplying by $b$ shows that $c$ satisfies
\eqref{eq:cquad}. The M\"obius map $t\mapsto c$ above has determinant $b^2(b-1)\neq0$, so it
is injective; the two values of $t$ are distinct when $(b+2)^2-4b=b^2+4\neq0$.
\end{proof}

\begin{example}\label{ex:rational}
Take $b=-\frac83$, $t=\frac12$, $u=1$, $a=\frac12$; then $bt^2+(b+2)t+1=0$ and
$u^5t(2t+1)=1$. Thus
\[
A=\mat{1}{1/2}{1}{0},\qquad B=\mat{1}{0}{0}{-8/3}\qquad\text{satisfy}\qquad
ABA^3BA=\mat{-2/9}{0}{0}{1}.
\]
Theorem~\ref{thm:criterion} applies with $\K=\R$, $u=1$, $a=\frac12$, $b=-\frac83$ and
$c=-\frac29$: indeed, if $\ln(2/9)/\ln(8/3)=m/n$ with integers $m,n$, then
$2^{n}3^{-2n}=2^{3m}3^{-m}$, so $n=3m$ and $2n=m$, i.e.\ $m=n=0$. Hence $\gen{A,B}$ is dense in
$\M(\R)$. In particular, the minimal number of generators of a dense subsemigroup of $\M(\R)$
is $2$.
\end{example}

Example~\ref{ex:rational} is one member of a one-parameter family. It is convenient to
parametrize $b<0$ by $s>0$ via
\[
b=\frac{s^2-4}{2s},\qquad \sqrt{b^2+4}=\frac{s^2+4}{2s};
\]
the condition $b<-1$ is equivalent to $0<s<\sqrt5-1$, and Example~\ref{ex:rational} is
$s=\frac23$.

\begin{theorem}\label{thm:realfamily}
For $0<s<\sqrt5-1$ put
\[
b=\frac{s^2-4}{2s},\qquad
u=\Bigl(\frac{(2-s)^2}{s(s+2)}\Bigr)^{1/5},\qquad
a=\frac{s\,u^2}{2-s},\qquad
c=-\frac{s(2-s)}{4},
\]
and $A=\smat{u}{a}{1}{0}$, $B=\diag(1,b)$. Then $ABA^3BA=\diag(c,1)$, $b<-1$ and
$-\frac14\le c<0$. If
\begin{equation}\label{eq:rhoirr}
\rho(s)=\frac{\ln\bigl(s(2-s)/4\bigr)}{\ln\bigl((4-s^2)/(2s)\bigr)}\notin\Q,
\end{equation}
then $\gen{A,B}$ is dense in $\M(\R)$. Condition \eqref{eq:rhoirr} fails for at most countably
many $s$. For example it holds for $s=\frac23$ (Example~\ref{ex:rational}) and for $s=1$, which
gives $b=-\frac32$, $c=-\frac14$, $u=3^{-1/5}$, $a=3^{-2/5}$.
\end{theorem}

\begin{proof}
Put $t=s/(2-s)>0$. One checks that $bt^2+(b+2)t+1=0$ and $t(2t+1)=s(s+2)/(2-s)^2$, so
$u^5t(2t+1)=1$ and $a=u^2t$. By Proposition~\ref{prop:word}, $ABA^3BA=\diag(c',1)$ with
$c'=b\bigl((b^2+2b+2)t+b+1\bigr)/\bigl((b+4)t+2\bigr)$, and substituting $b$ and $t$ gives
$c'=-s(2-s)/4=c$. Clearly $|c|\le\frac14$ (with equality at $s=1$), and $b<-1$ for
$0<s<\sqrt5-1$. Thus $\gen{A,B}$ contains $A$, $B$ and $C=\diag(c,1)$, and
Theorem~\ref{thm:criterion} applies when $\rho(s)=\ln|c|/\ln|b|\notin\Q$.

The function $\rho$ is real analytic on $(0,\sqrt5-1)$, tends to $-1$ as $s\to0^+$ and to
$-\infty$ as $s\to(\sqrt5-1)^-$ (where $|b|\to1^+$ while $|c|$ stays bounded away from $1$).
Hence $\rho$ is not constant, each level set $\rho^{-1}(q)$ is discrete, hence countable, and
$\rho^{-1}(\Q)$ is countable. For $s=1$, $\rho=\ln(1/4)/\ln(3/2)$, and
$4^{-n}=(3/2)^{m}$ forces $m=n=0$.
\end{proof}

\begin{remark}
In Theorem~\ref{thm:realfamily} the other root $t<0$ of $bt^2+(b+2)t+1=0$ yields
$c=b/c'$ with $|c|>1$; in that case Theorem~\ref{thm:criterion} does not apply directly.
\end{remark}

We now turn to $\C$. For $|b|>3$, Rouch\'e's theorem applied on $|z|=1$ to
$z^2-(b^2-b+2)z+b$ and $-(b^2-b+2)z$ (note $|b^2-b+2|\ge|b|^2-|b|-2>1+|b|$) shows that
\eqref{eq:cquad} has exactly one root $c(b)$ in the open unit disk, and it is simple. Hence
$c(b)$ is a holomorphic function of $b$ on $\{|b|>3\}$. Let $t(b)$ be the root of
$bt^2+(b+2)t+1=0$ corresponding to $c(b)$ under Proposition~\ref{prop:word}, explicitly
\[
t(b)=\frac{2c(b)-b(b+1)}{b(b^2+2b+2)-(b+4)c(b)} .
\]

\begin{theorem}\label{thm:complexfamily}
Let $b\in\C\setminus\R$ with $|b|>3$, let $c=c(b)$ and $t=t(b)$ be as above, let $u$ be any
fifth root of $1/(t(2t+1))$, and put $a=u^2t$. Then $A=\smat{u}{a}{1}{0}$ and
$B=\diag(1,b)$ satisfy $ABA^3BA=\diag(c,1)$ with $0<|c|<1$, and $\gen{A,B}$ is dense in
$\M(\C)$ provided that
\begin{equation}\label{eq:H2b}
1,\quad \rho(b)=\frac{\ln|c(b)|}{\ln|b|},\quad \frac{\arg c(b)}{2\pi}-\frac{\arg b}{2\pi}\,\rho(b)
\quad\text{are linearly independent over }\Q .
\end{equation}
The set $\mathcal E$ of $b\in\C\setminus\R$ with $|b|>3$ for which \eqref{eq:H2b} fails meets
every ray $\{re^{i\varphi}:r>3\}$, $\varphi\notin\pi\Z$, in a countable set. In particular
$\mathcal E$ has Lebesgue measure zero, and there are uncountably many pairwise distinct
$2$-generator dense subsemigroups of $\M(\C)$.
\end{theorem}

\begin{proof}
The first assertions follow from Proposition~\ref{prop:word} ($c\neq0$ because $c\,c'=b$
for the other root $c'$), and density follows from Theorem~\ref{thm:criterion} with
$\K=\C$, since (H1) holds and \eqref{eq:H2b} is (H2).

Fix $\varphi\notin\pi\Z$ and write $b=re^{i\varphi}$, $\beta=\varphi/2\pi$. On $(3,\infty)$
the function $r\mapsto c(re^{i\varphi})$ is real analytic and nonvanishing, so $\rho$ and a
continuous branch $\gamma(r)$ of $\arg c/2\pi$ are real analytic. Other branches change
$\gamma-\beta\rho$ by integers, which does not affect \eqref{eq:H2b}. For a nonzero triple
$(m_0,m_1,m_2)\in\Z^3$ let
\[
G(r)=m_0+m_1\rho(r)+m_2\bigl(\gamma(r)-\beta\rho(r)\bigr).
\]
We show $G\not\equiv0$. Then its zero set in $(3,\infty)$ is discrete, hence countable, and
$\mathcal E\cap\{re^{i\varphi}\}$, a countable union of such sets, is countable. The measure
statement follows from Fubini's theorem in polar coordinates.

\emph{Asymptotics.} Let $c'=b/c$ be the other root of \eqref{eq:cquad}, so
$c+c'=b^2-b+2$. Since $|c|<1$, $|c'|\ge|b|^2-|b|-3$ and $bc=b^2/c'$ is bounded, so
$c=O(|b|^{-1})$. Hence
\[
bc=\frac{b^2}{b^2-b+2-c}=\Bigl(1-\frac1b+\frac2{b^2}+O(|b|^{-3})\Bigr)^{-1},\qquad
\log(bc)=\frac1b-\frac{3}{2b^2}+O(|b|^{-3}).
\]
Taking real and imaginary parts (with the branch of $\arg$ for which $\arg(bc)\to0$),
\begin{align*}
\ln|c|&=-\ln r+\frac{\cos\varphi}{r}-\frac{3\cos2\varphi}{2r^2}+O(r^{-3}),\\
2\pi\gamma&=-\varphi-\frac{\sin\varphi}{r}+O(r^{-2}),
\end{align*}
so that
\[
\rho=-1+\frac{\cos\varphi}{r\ln r}-\frac{3\cos2\varphi}{2r^2\ln r}+O\Bigl(\frac{1}{r^3\ln r}\Bigr),
\qquad
\gamma-\beta\rho=-\frac{\sin\varphi}{2\pi r}-\frac{\beta\cos\varphi}{r\ln r}+O(r^{-2}).
\]
\emph{Conclusion.} Suppose $G\equiv0$. Letting $r\to\infty$ gives $m_0=m_1$, so
$m_1(\rho+1)+m_2(\gamma-\beta\rho)\equiv0$. Multiplying by $r$ and letting $r\to\infty$ gives
$-m_2\sin\varphi/2\pi=0$, hence $m_2=0$. Then $m_1(\rho+1)\equiv0$. But $\rho+1\not\equiv0$:
if $\cos\varphi\neq0$ then $r\ln r\,(\rho+1)\to\cos\varphi$, and if $\cos\varphi=0$ then
$\cos2\varphi=-1$ and $r^2\ln r\,(\rho+1)\to\frac32$. Hence $m_1=0$ and $m_0=0$, a contradiction.
\end{proof}

\begin{remark}
For $b=ri$, $r>3$, Theorem~\ref{thm:complexfamily} gives a dense $2$-generator subsemigroup of
$\M(\C)$ for all $r$ outside a countable set, for each of the five choices of $u$. The value
$c(ri)$ is the root of \eqref{eq:cquad} with modulus less than $1$; with an appropriate
choice of square root it equals
$\frac12\bigl(2+b^2+\sqrt{4+b^2}-b-b\sqrt{4+b^2}\bigr)$, and $c(ri)=-i/r+O(r^{-2})$.
\end{remark}

\section{Remarks and open problems}\label{sec:remarks}

\begin{enumerate}[leftmargin=*,label=(\arabic*)]
\item What is the minimal number of generators of a
dense subsemigroup of the cone of nonnegative $n\times n$ matrices?
\item In Theorem~\ref{thm:realfamily}, is $\gen{A,B}$ dense also for the
(countably many) $s$ with $\rho(s)\in\Q$? In that case $\{b^kc^l\}$ is not dense in $\R$, so Step~2 of
Theorem~\ref{thm:criterion} breaks down, but other words in $A,B$ might supply the missing
scalars.
\item Is the set of pairs $(A,B)\in\M(\K)^2$ generating a dense subsemigroup of $\M(\K)$
of full Lebesgue measure? Is it dense? Our families show only that it is uncountable and
contains a real-analytic curve (minus a countable set) when $\K=\R$.
\item Theorem~\ref{thm:closure} computes $\cl{\gen{f,g}}$ for the pair $(R,S)$. It would be
interesting to compute the closures for the other pairs in Theorem~\ref{thm:classification},
which are not dense but do have dense orbits on $(0,\infty)$.
\end{enumerate}

\end{document}